\documentclass[11pt]{amsart}

\usepackage[a4paper,hmargin=3.5cm,vmargin=4cm]{geometry}
\usepackage{amsfonts,amssymb,amscd,amstext}
\usepackage{graphicx}
\usepackage[dvips]{epsfig}
\usepackage{pstricks,pst-plot}
\usepackage{color}

\usepackage{fancyhdr}
\input xy
\xyoption{all}

\usepackage{times}

\usepackage{enumerate}
\usepackage{titlesec}
\usepackage{mathrsfs}

\titleformat{\section}
{\filcenter\bfseries\large} {\thesection{.}}{0.2cm}{}
\titleformat{\subsection}[runin]
{\bfseries} {\thesubsection{.}}{0.15cm}{}[.]
\titleformat{\subsubsection}[runin]
{\em}{\thesubsubsection{.}}{0.15cm}{}[.]

\usepackage[up,bf]{caption}

\newtheorem{theorem}{Theorem}[section]
\newtheorem{proposition}[theorem]{Proposition}

\newtheorem{lemma}[theorem]{Lemma}
\newtheorem{corollary}[theorem]{Corollary}

\theoremstyle{definition}
\newtheorem{definition}[theorem]{Definition}
\newtheorem{remark}[theorem]{Remark}

\newtheorem{problem}[theorem]{Problem}

\numberwithin{equation}{section}
\numberwithin{figure}{section}

\newcommand\Bcal{\mathcal{B}}

\newcommand\Cscr{\mathscr{C}}

\newcommand\Oscr{\mathscr{O}}

\newcommand\B{\mathbb{B}}
\newcommand\C{\mathbb{C}}

\newcommand\N{\mathbb{N}}

\newcommand\R{\mathbb{R}}

\newcommand\igot{\mathfrak{i}}

\renewcommand\igot{\mathfrak{i}}

\renewcommand\imath{\igot}

\newcommand\wh{\widehat}

\newcommand\hra{\hookrightarrow}

\newcommand\Psh{\mathrm{Psh}}

\newcommand\Aut{\mathrm{Aut}}

\newcommand\Id{\mathrm{Id}}

\newcommand\Wprocess{Wold process}

\begin{document}

\fancyhead[LO]{A long $\C^2$ without holomorphic functions}
\fancyhead[RE]{Luka Boc Thaler and Franc Forstneri\v c}
\fancyhead[RO,LE]{\thepage}

\thispagestyle{empty}

\vspace*{7mm}
\begin{center}
{\bf \LARGE A long $\C^2$ without holomorphic functions}
\vspace*{5mm}

{\large\bf Luka Boc Thaler and Franc Forstneri\v c}
\end{center}


\vspace*{7mm}

\begin{quote}
{\small
\noindent {\bf Abstract}\hspace*{0.1cm}
In this paper we construct for every integer $n>1$ a complex manifold  of dimension $n$ which
is exhausted by an increasing sequence of biholomorphic images of $\C^n$
(i.e., a {\em long $\C^n$}), but it does not admit any nonconstant holomorphic 
or plurisubharmonic functions (see Theorem \ref{th:main}). Furthermore, we introduce new holomorphic 
invariants of a complex manifold $X$, the {\em stable core} and the {\em strongly stable core}, 
%
%
that are based on the long term behavior of hulls of compact sets with respect to an exhaustion of $X$.
We show that every compact polynomially convex set $B\subset \C^n$ such that $B=\overline{\mathring B}$
is the strongly stable core of a long $\C^n$; in particular, holomorphically nonequivalent sets 
give rise to nonequivalent long $\C^n$'s (see Theorems \ref{th:XB} and \ref{th:SC} (a)).
Furthermore, for every open set $U\subset \C^n$ there exists a long $\C^n$ whose stable core is dense in $U$  (see Theorem \ref{th:SC} (b)).
%
%
It follows that for any $n>1$ there is a continuum of  pairwise nonequivalent long $\C^n$'s
with no nonconstant plurisubharmonic functions and no nontrivial holomorphic automorphisms.  
These results answer several long standing open problems.

\vspace*{0.1cm}
\noindent{\bf Keywords}\hspace*{0.1cm} Holomorphic function, Stein manifold, long $\C^n$, 
Fatou--Bieberbach domain, Chern--Moser normal form

\vspace*{0.1cm}

\noindent{\bf MSC (2010):}\hspace*{0.1cm}} 32E10, 32E30, 32H02

\date{\today}

\end{quote}

%
%
%
%
\section{Introduction}\label{sec:intro}
A complex manifold $X$ of dimension $n$ is said to be a {\em long $\C^n$}
if it is the union of an increasing sequence of domains 
$X_1\subset X_2\subset X_3\subset \cdots \subset \bigcup_{j=1}^\infty X_j=X$ such 
that each $X_j$ is biholomorphic to the complex Euclidean space $\C^n$. It is immediate that any long $\C$
is biholomorphic to $\C$. However, for $n>1$, this class of complex manifolds is still very mysterious.
The long standing question, whether there exists a long $\C^n$ which is not biholomorphic to $\C^n$, 
was answered in 2010 by E.\ F.\ Wold (see \cite{Wold2010}) who constructed a long $\C^n$ 
that is not holomorphically convex, hence not a Stein manifold.  
Wold's construction is based on his examples of non-Runge Fatou--Bieberbach domains in $\C^n$ 
(see \cite{Wold2008}; an exposition of both results can be found in \cite[Section 4.20]{Forstneric2011}).
In spite of these interesting examples, the theory has not been developed since.
In particular, it remained unknown whether there exist long $\C^2$'s without nonconstant holomorphic
functions, and whether there exist at least two nonequivalent non-Stein long $\C^2$'s.

We begin with the following result which answers the first question affirmatively.

%
%
\begin{theorem}\label{th:main}
For every integer $n>1$ there exists a long $\C^n$ without any nonconstant holomorphic 
or plurisubharmonic functions.
\end{theorem}

%
%

Theorem \ref{th:main} is proved in Section \ref{sec:proof1}.
It contributes to the line of counterexamples to the classical Union Problem for Stein manifolds: 
{\em is an increasing union of Stein manifolds always Stein?} 
For domains in $\C^n$ this question was raised by Behnke and Thullen  in 1933 (see \cite{BehnkeThullen1933}), 
and they gave an affirmative answer in 1939 (see \cite{BehnkeStein1939}).
Some progress on the general question was made by Stein \cite{Stein1956} and 
Docquier and Grauert \cite{DocquierGrauert1960}. The first counterexample 
was given in any dimension $n\ge 3$  by J.\ E.\ Forn\ae ss in 1976 (cf.\ \cite{Fornaess1976});
he found an increasing union of balls that is not holomorphically convex, hence not Stein. 
The key ingredient in his proof is a construction of a biholomorphic map $\Phi\colon \Omega\to \Phi(\Omega)\subset \C^3$ 
on a bounded neighborhood $\Omega\subset \C^3$ of any compact set $K\subset \C^3$ with nonempty  interior 
such that the polynomial hull of $\Phi(K)$ is not contained in $\Phi(\Omega)$. 
(A phenomenon of this type was first described by Wermer in 1959; see \cite{Wermer1959}.)
In 1977, Forn{\ae}ss and Stout constructed an increasing union of three-dimensional polydiscs 
without nonconstant holomorphic functions (see \cite{FornaessStout1977}). 
Increasing unions of hyperbolic Stein manifolds were studied further by 
Forn{\ae}ss and Sibony (see \cite{FornaessSibony1981}) and Forn{\ae}ss (see \cite{Fornaess2004}). 
The first counterexample to the union problem in dimension $n=2$ was the result 
of Wold (cf.\ \cite{Wold2010}) on the existence of a non-Stein long $\C^2$. 
For the connection with Bedford's conjecture, see the survey   \cite{Abbondandolo2013} by Abbondandolo
et al. 

Another question that has been asked repeatedly over a long period of time 
is whether there exist infinitely many nonequivalent long $\C^n$'s for any or all $n>1$. 
Up to now, only two different long $\C^2$'s have been known, namely the standard $\C^2$ and 
a non-Stein long $\C^2$ constructed by Wold \cite{Wold2010}. In higher dimension $n>2$
one can get a few more examples by considering Cartesian products of long $\C^k$'s
for different values of $k$.  In this paper, we introduce new biholomorphic invariants
of a complex manifold, the {\em stable core} and the {\em strongly stable core} (see Definition \ref{def:SC}), 
which allow us to distinguish certain 
pairs of long $\C^n$'s one from another.  In our opinion, this is the main new contribution of the paper
from the conceptual point of view.  
With the help of these invariants, we answer the above mentioned question 
affirmatively by proving the following result. 

%
%

Recall that a compact subset $B$ of a topological space $X$ is said to be {\em regular} if it is the closure of its interior,
$B=\overline{\mathring B}$.

%
%
%
%
\begin{theorem}\label{th:XB}
Let $n>1$. To every regular compact polynomially convex set $B\subset \C^n$ 
we can associate a complex manifold $X(B)$, 
which is a long $\C^n$ containing a biholomorphic copy of $B$, 
such that every biholomorphic map $\Phi\colon X(B)\to X(C)$ between two such
manifolds takes $B$ onto $C$. In particular, for every holomorphic automorphism
$\Phi\in\Aut(X(B))$, the restriction $\Phi|_B$ is an automorphism of $B$. 
%
%
We can choose $X(B)$ such that it has no nonconstant holomorphic or plurisubharmonic functions.
\end{theorem}

It follows that the manifold $X(B)$ can be biholomorphic to $X(C)$ only if $B$ is biholomorphic to $C$.
Our construction likely gives many nonequivalent long $\C^n$'s associated to the same set $B$. 
A more precise result is given by Theorem \ref{th:SC} below.

By considering the special case when $B$ is the closure of a strongly pseudoconvex domain, 
Theorem \ref{th:XB} shows that the moduli space of long $\C^n$'s
contains the moduli space of germs of smooth strongly pseudoconvex real hypersurfaces 
in $\C^n$. This establishes a surprising connection between  long $\C^n$'s and the Cauchy-Riemann geometry.
It has been known since Poincar\'e's paper \cite{Poincare1907} 
that most pairs of smoothly bounded strongly pseudoconvex domains in $\C^n$ 
are not biholomorphic to each other, at least not by maps extending smoothly to
the closed domains. It was shown much later by C.\ Fefferman 
\cite{Fefferman1974} that the latter condition is is automatically fulfilled. 
(For elementary proofs of Fefferman's theorem, see Pinchuk and Khasanov \cite{PinchukKhasanov1989MS} 
and Forstneri\v c \cite{Forstneric1992EM}.) 
A complete set of local holomorphic invariants of a strongly pseudoconvex real-analytic hypersurface
is provided by the {\em Chern-Moser normal form}; see \cite{ChernMoser1974}. 
Most such domains have no holomorphic automorphisms other than the identity map. 
(For surveys of this topic, see  e.g.\ \cite{BaouendiEbenfeltRothschild1999} and 
\cite{Forstneric1993}.) Hence, Theorem  \ref{th:XB} implies the following corollary.

%
%

%
%
%
%
\begin{corollary}\label{cor:infinitely}
For every $n>1$ there is
%
%
%
a continuum of pairwise nonequivalent long $\C^n$'s with no nonconstant holomorphic or plurisubharmonic 
functions and no nontrivial holomorphic automorphisms.  
\end{corollary}

%
%

We now describe the new biholomorphic invariants alluded to above, the {\em stable core} and the {\em strongly stable core}
of a complex manifold. Their definition  is based on the following property which a compact set in a 
complex manifold may or may not have. Given a pair of compact sets $K\subset L$ in a complex manifold $X$, we write
\begin{equation}\label{eq:hull}
	\wh K_{\Oscr(L)}=\{x\in L: |f(x)|\le \sup_K |f| \ \text{for all $f\in\Oscr(L)$}\},
\end{equation}
where $\Oscr(L)$ is the algebra of holomorphic functions on neighborhoods of $L$.

%
%
%
%
\begin{definition}[The Stable Hull Property] \label{def:SHP}
A compact set $K$ in a complex manifold $X$ has the  {\em stable hull property}, SHP, if there exists an exhaustion
$K_1\subset K_2\subset \cdots \subset \bigcup_{j=1}^\infty K_j=X$ by compact sets 
such that $K\subset K_1$, $K_j\subset \mathring K_{j+1}$ for every $j\in\N$, 
and the increasing sequence of hulls $\wh K_{\Oscr(K_j)}$ stabilizes, i.e., there is a $j_0\in\N$ such that 
\begin{equation}\label{eq:stable}
  	{\wh K}_{\Oscr(K_j)}= {\wh K}_{\Oscr(K_{j_0})}\quad  \text{for all $j\ge j_0$}.
\end{equation}
\end{definition}

Obviously, SHP is a biholomorphically invariant property: if a compact set $K \subset X$ satisfies 
condition \eqref{eq:stable} with respect to some exhaustion $(K_j)_{j\in\N}$ of $X$,
then for any biholomorphic map $F\colon X\to Y$ the set $F(K) \subset Y$ satisfies 
\eqref{eq:stable} with respect to the exhaustion $L_j=F(K_j)$ of $Y$.
What is less obvious, but needed to make this condition useful, is its 
independence of the choice of the exhaustion; see Lemma \ref{lem:SHP}.

%
%
\begin{definition} \label{def:SC}
Let $X$ be a complex manifold. 
\begin{itemize}
\item[\rm (i)] The {\em stable core} of $X$, denoted $SC(X)$, 
is the open set consisting of all points $x\in X$ which admit a compact neighborhood 
$K\subset X$ with the stable hull property. 
\vspace{1mm}
\item[\rm (ii)] 
%
%
A regular compact set $B\subset X$ is called the {\em strongly stable core} of $X$, denoted $SSC(X)$, 
if $B$ has the stable hull property, but no compact set $K \subset X$ with
$\mathring K \setminus B \ne \emptyset$ has the stable hull property.
\end{itemize}
\end{definition}

Clearly, the stable core always exists and is a biholomorphic invariant,
in the sense that any biholomorphic map $X\to Y$ maps $SC(X)$ onto $SC(Y)$. 
In particular, every holomorphic automorphism of $X$ maps the stable core $SC(X)$ onto itself. 
%
%
The strongly stable core $SSC(X)$ need not exist in general; if it does, then its interior equals the stable core $SC(X)$
and $SSC(X)=\overline{SC(X)}$.  In part (ii) we must restrict attention to {\em regular} compact sets since otherwise 
the definition would be ambiguous.


\begin{theorem} \label{th:SC}
Let $n>1$.
\begin{itemize}
\item[\rm (a)]  For every regular compact  polynomially convex set $B\subset \C^n$ (i.e., $B=\overline{\mathring B}$) 
there exists a long $\C^n$, $X(B)$, which admits no nonconstant plurisubharmonic functions and
whose strongly stable core equals $B$: $SSC(X(B))=B$.
\vspace{1mm}
\item[\rm (b)] For every open set $U\subset\C^n$ there exists a long $\C^n$, $X$, 
which admits no nonconstant holomorphic functions and
satisfies $SC(X)\subset U$ and $\overline U=\overline{SC(X)}$.
\end{itemize}
\end{theorem}
 
%
%
In Theorem \ref{th:SC} we have identified the sets $B,U \subset \C^n$ with their
images in the long $\C^n$, $X=\bigcup_{k=1}^\infty X_k$, by identifying $\C^n$ with 
the first domain $X_1\subset X$.

Assuming Theorem \ref{th:SC}, we now prove Theorem \ref{th:XB}.

%
%
%

\begin{proof}[Proof of Theorem \ref{th:XB}]
Let $B$ be a regular compact polynomially convex set in $\C^n$ for some $n>1$. By Theorem \ref{th:SC}
there exists a long $\C^n$, $X=X(B)$, whose strongly stable core is $B$.
Assume that $F \in \Aut(X)$. Then, $F(B)$ has SHP (see Definition \ref{def:SHP}).
Since $B$ is the biggest regular compact subset of $X$ with SHP (see (ii) in Definition \ref{def:SC}), we have that 
$\Phi(B)\subset B$. Applying the same argument to the inverse automorphism $\Phi^{-1}\in\Aut(X)$ gives 
$\Phi^{-1}(B)\subset B$, and hence $B\subset \Phi(B)$. Both properties together 
imply that $\Phi(B)=B$, and hence $\Phi|_B\in \Aut(B)$. 

In the same way, we see that a biholomorphic map $X(B)\to X(C)$
between two long $\C^n$'s, furnished by part (a) in Theorem \ref{th:SC}, maps
$B$ biholomorphically onto $C$. Hence, if $B$ is not biholomorphic to 
$C$, then $X(B)$ is not biholomorphic to $X(C)$.
\end{proof}

Theorem \ref{th:SC} is proved in Section \ref{sec:XB}.
We construct  manifolds with these properties by improving the recursive procedure devised by 
Wold \cite{Wold2008,Wold2010}. The following key ingredient was introduced in \cite{Wold2008}; 
it will henceforth be called the {\em \Wprocess} (see Remark \ref{rem:Wprocess}):
 
Given a compact holomorphically convex set $L\subset \C^*\times \C^{n-1}$ with nonempty interior, 
there is a holomorphic automorphism  $\psi\in\Aut(\C^*\times \C^{n-1})$ 
such that the polynomial hull  $\wh{\psi(L)}$ of the set $\psi(L)$ intersects the hyperplane $\{0\}\times \C^{n-1}$. 
By precomposing $\psi$ with a suitably chosen Fatou--Bieberbach map $\theta\colon \C^n\hra \C^*\times \C^{n-1}$, 
we obtain a Fatou--Bieberbach map $\phi=\psi\circ\theta\colon\C^n\hra\C^n$ such that, for a given polynomially convex set 
$K\subset\C^n$ with nonempty interior, we have that $\wh{\phi(K)}\not\subset\phi(\C^n)$.

At every step of the recursion we perform the  \Wprocess\  simultaneously on finitely many pairwise disjoint
compact sets $K_1,\ldots, K_m$ in the complement of the given regular polynomially convex set $B\subset \C^n$, chosen such
that $\bigcup_{j=1}^m K_j\cup B$ is polynomially convex, thereby ensuring that polynomial hulls 
of their images $\phi(K_j)$ escape from the range of the injective holomorphic map 
$\phi\colon \C^n\hra \C^n$ constructed in the recursive step. At the same time, we ensure 
that $\phi$ is close to the identity map on a neighborhood of $B$, and hence the image $\phi(B)$ remains polynomially convex. 
In practice, the sets $K_j$ will be small pairwise disjoint closed balls  in the complement of $B$ 
whose number will increase during the process. We devise the process so that every point in a 
certain countable dense set $A=\{a_j\}_{j=1}^\infty \subset X\setminus B$ 
is the center of a decreasing sequence of balls whose $\Oscr(X_k)$-hulls escape
from each compact set in $X$; hence none of these balls has the stable hull property. 
This implies that $B$ is  the strongly stable core of $X$. 

To prove part (b),  we modify the recursion by introducing a new small ball 
$B'\subset U\setminus B$ at every stage. Thus, the set $B$ acquires additional 
connected components during the recursive process.
The sequence of added balls $B_l$ is chosen such that their union is dense in the given 
open subset $U\subset \C^n$, while the sequence of sets $K_j$ on which the \Wprocess\ is performed
densely fills the complement $X\setminus \overline U$. It follows that the stable core of the 
limit manifold $X=\bigcup_{k=1}^\infty X_k$ is contained in $U$ and is everywhere dense in $U$.

By combining the technique used in the proof of Theorem \ref{th:main} (see Section \ref{sec:proof1})
with those in \cite[proof of Theorem 1.1]{Forstneric2012}, one can easily obtain the following 
result for holomorphic families of long $\C^n$'s.
(Compare with \cite[Theorem 1.1]{Forstneric2012}.) We leave out the details.

%
%
%
%
\begin{theorem}\label{th:main2}
Let $Y$ be a Stein manifold, and let $A$ and $B$ be disjoint finite or countable sets in $Y$. 
For every integer $n>1$ there exists a complex manifold $X$ of dimension $\dim Y +n$ and a 
surjective holomorphic submersion $\pi\colon X\to Y$ with the following properties:
\begin{itemize}
\item the fiber $X_{y}=\pi^{-1}(y)$ over any point $y\in Y$ is a long $\C^n$;
\item $X_{y}$ is biholomorphic to $\C^n$ if $y\in A$;
\item $X_{y}$ does not admit any nonconstant plurisubharmonic function if $y\in B$.
\end{itemize}
If the base $Y$ is $\C^p$, then $X$ may be chosen to be a long $\C^{p+n}$.
\end{theorem}

Note that one or both of the sets $A$ and $B$ in Theorem \ref{th:main2} may be chosen 
everywhere dense in $Y$. The same result holds if $A$ is a union of at most countably many 
closed complex subvarieties of $Y$ and the set $B$ is countable.

%
%
%
%
Several interesting questions on long $\C^n$'s remain open; we record some of them.

\begin{problem}
(A) Does there exist a long $\C^2$ which admits a nonconstant holomorphic function, but is not Stein?

\noindent  (B) To what extent is it possible to prescribe the algebra $\Oscr(X)$ of a long $\C^n$?

\noindent (C) Does there exist a long $\C^n$ for any $n>1$ which is a Stein manifold different from $\C^n$?

%
%
\noindent (D) Does there exist a long $\C^n$ without nonconstant meromorphic functions?

\noindent (E) What can be said about the (non) existence of complex anaytic subvarieties of positive 
dimension in non-Stein long $\C^n$'s? 
%
%
%
\end{problem}

In dimensions $n>2$, an affirmative answer to Problem (A) is provided 
by the product $X=\C^p\times X^{n-p}$ for any $p=1,\ldots, n-2$, where $X^{n-p}$ is a long 
$\C^{n-p}$ without nonconstant holomorphic functions, furnished by Theorem \ref{th:main}. 
Note that $\Oscr(\C^p\times X^{n-p})\cong \Oscr(\C^p)$ is the algebra of functions coming from the base.
Indeed, any example furnished by Theorem \ref{th:main2}, with the base $Y=\C^p$
$(p\ge 1)$ and $B$ dense in $\C^p$, is of this kind.

%
%
Regarding question (D), note that the Fatou-Bieberbach maps $\phi_k\colon \C^n\hra\C^n$ used in our constructions
have rationally convex images, in the sense that for any compact polynomially convex set $K\subset \C^n$ its
image $\phi_k(K)$ is a rationally convex set in $\C^{n}$; this gives rise to nontrivial 
meromorphic functions on the resulting long $\C^n$'s.

Since every long $\C^n$ is an {\em Oka manifold} (cf.\ \cite{Larusson2010} and
\cite[Proposition 5.5.6, p.\ 200]{Forstneric2011}),
the results of this paper also contribute to our understanding of the class of Oka manifolds,
that is, manifolds which are the most natural targets for holomorphic maps from Stein manifolds
and reduced Stein spaces. 

Note that every long $\C^n$ is a topological cell according to a theorem of Brown \cite{Brown1961}. 
%
%
Furthermore, it was shown by Wold \cite[Theorem 1.2]{Wold2010} that, 
if $X= \bigcup_{k=1}^\infty X_k$ is a long $\C^n$ and $(X_k,X_{k+1})$ is a Runge pair for every $k\in\N$, 
then $X$ is biholomorphic to $\C^n$. Since the Runge property always holds in the $\Cscr^\infty$ category,
i.e. for smooth diffeomorphisms of Euclidean spaces, his proof can be adjusted to show that every long
$\C^n$ is also diffeomorphic to $\R^{2n}$. Hence, Theorems \ref{th:XB} and \ref{th:SC} imply the following corollary.

\begin{corollary} \label{cor:Oka}
For every $n>1$ there exists a continuum of pairwise nonequivalent Oka manifolds 
of complex dimension $n$ which are all diffeomorphic to $\R^{2n}$. 
\end{corollary}

In Section \ref{sec:exhaustion} we show that  $\C^n$ for any $n>1$ can also be represented 
as an increasing union of non-Runge Fatou--Bieberbach domains.


\section{Preliminaries}
\label{sec:prel}

In this section, we introduce the notation and recall the basic ingredients. 

We denote by $\Oscr(X)$ the algebra of all holomorphic functions on a complex manifold $X$.
For a compact set $K\subset X$, $\Oscr(K)$ stands for the algebra of functions
holomorphic in open neighborhoods of $K$ (in the sense of germs along $K$). 
The {\em $\Oscr(X)$-convex hull} of $K$ is 
\[
	\wh K_{\Oscr(X)}= \{x\in X : |f(x)| \le \sup_K |f| \ \ \text{for all $f\in \Oscr(X)$} \}.
\] 
When $X=\C^n$, the set $\wh K=\wh K_{\Oscr(\C^n)}$ is the {\em polynomial hull} of $X$.
If $\wh K_{\Oscr(X)}=K$, we say that $K$ is {\em holomorphically convex in $X$};
if $X=\C^n$ then $K$ is polynomially convex.
More generally, if $K\subset L$ are compact sets in $X$, we define the 
hull $\wh K_{\Oscr(L)}$ by \eqref{eq:hull}.

Given a point $p\in\C^n$, we denote by $\B(p;r)$ the closed ball  of radius $r$ centered at $p$. 

We shall frequently use the following basic result; see e.g.\ Stout \cite{Stout1971,Stout2007} for the first part 
(which is a simple application of E.\ Kallin's lemma) and \cite{Forstneric1986} for the second part.

\begin{lemma}\label{lem:stability}
Assume that $B\subset \C^n$ is a compact polynomially convex set.
For any finite set $p_1,\ldots,p_m\in\C^n\setminus B$ and for all sufficiently
small numbers $r_1>0,\ldots,r_m>0$, the set $\bigcup_{j=1}^m \B(p_j,r_j) \cup B$
is polynomially convex. Furthermore, if $B$ is the closure of a bounded
strongly pseudoconvex domain with $\C^2$ boundary, then any sufficiently
$\Cscr^2$-small deformation of $B$ in $\C^n$ is still polynomially convex.
\end{lemma}

The key ingredient in our proofs will be the main result of the Anders\'en--Lempert theory 
%
%
as formulated by Forstneri\v c and Rosay \cite[Theorem 1.1]{ForstnericRosay1993};
see Theorem \ref{th:AL} below. We shall use it not only for $\C^n$, but also for $X=\C^*\times \C^{n-1}$.
The result holds for any Stein manifold which enjoys the following 
{\em density property} introduced by Varolin \cite{Varolin2001}. (See also
\cite[Definition 4.10.1]{Forstneric2011}.)

\begin{definition}\label{def:DP}
A complex manifold $X$ enjoys the (holomorphic) {\em density property} if
every holomorphic vector field on $X$ can be approximated, uniformly on compacts,
by Lie combinations (sums and commutators) of $\C$-complete  holomorphic vector fields on $X$.
\end{definition}

By Anders\'en and Lempert (see \cite{Andersen1990,AndersenLempert1992}), the complex
Euclidean space $\C^n$ for $n>1$ enjoys the  density property. More generally,
Varolin proved that any complex manifold $X=(\C^*)^k\times \C^l$ 
with $k+l\ge 2$ and $l\ge 1$ enjoys the density property (cf.\ \cite{Varolin2001}). For surveys of this subject,
see for instance \cite[Chapter 4]{Forstneric2011} and \cite{KalimanKutzschebauch2011}.

%
%

%
%
%
%
\begin{theorem} \label{th:AL}
Let $X$ be a Stein manifold with the density property, and let
\[
	\Phi_t\colon \Omega_0 \longrightarrow \Omega_t=\Phi_t(\Omega_0)\subset X, 
	\quad t\in [0,1]
\] 
be a smooth isotopy  of biholomorphic maps of $\Omega_0$ onto Runge domains 
$\Omega_t\subset X$ such that $\Phi_0=\mathrm{Id}_{\Omega_0}$.  
Then, the map $\Phi_1\colon \Omega_0\to\Omega_1$ can be approximated uniformly on compacts
in $\Omega_0$ by holomorphic automorphisms of $X$.
\end{theorem}

This is a version of \cite[Theorem 1.1]{ForstnericRosay1993} in which $\C^n$ is replaced by an 
arbitrary Stein manifold with the density property (see also \cite[Theorem 4.10.6]{Forstneric2011}).
For a detailed proof of Theorem \ref{th:AL}, see \cite[Theorem 1.1]{ForstnericRosay1993} for the case
$X=\C^n$ and \cite[Theorem 8]{Ritter2013} for the general case (which follows the one in
\cite{ForstnericRosay1993} essentially verbatim).

%
%
%
%
\section{Construction of a long $\C^n$ without holomorphic functions}
\label{sec:proof1}

In this section, we prove Theorem \ref{th:main}.
We begin by recalling the general construction of a long $\C^n$ (cf.\  \cite{Wold2010} or
\cite[Section 4.20]{Forstneric2011}).

Recall that a {\em Fatou--Bieberbach map} is an injective holomorphic map $\phi \colon \C^n\hra \C^n$ 
such that $\phi(\C^n)\subsetneq \C^n$; the image $\phi(\C^n)$ of such map is called a {\em Fatou--Bieberbach domain}.
Every complex manifold $X$ which is a long $\C^n$ is determined by a sequence of Fatou--Bieberbach 
maps $\phi_k \colon \C^n\to \C^n$ $(k=1,2,3,\ldots)$. 
The elements of $X$ are represented by infinite strings
$x=(x_i,x_{i+1},\ldots)$, where $i\in\N$ and for every $k=i,i+1,\ldots$ we have 
$x_k\in \C^n$ and $x_{k+1}=\phi_k(x_k)$.
Another string $y=(y_j,y_{j+1},\ldots)$ determines the same element of $X$ if and only if
one of the following possibilities holds:
\begin{itemize}
\item $i=j$ and $x_i=y_i$ (and hence $x_k=y_k$ for all $k> i$);
\vspace{1mm}
\item $i<j$ and $y_j=\phi_{j-1}\circ \cdots \circ \phi_i(x_i)$;
\vspace{1mm}
\item $j<i$ and $x_i=\phi_{i-1}\circ \cdots \circ \phi_j(y_j)$.
\end{itemize}
For each $k\in\N$, let $\psi_k \colon \C^n\hra X$ be the injective map
sending $z\in \C^n$ to the equivalence class of the string $(z,\phi_k(z),\phi_{k+1}(\phi_k(z)),\ldots)$. 
Set $X_k=\psi_k(\C^n)$ 
and let $\iota_k\colon X_k \hra X_{k+1}$ be the inclusion map induced by 
$(x_k,x_{k+1},x_{k+2},\ldots) \mapsto (x_{k+1},x_{k+2},\ldots)$. 
Then, 
\begin{equation}\label{eq:comm}
	\iota_k \circ \psi_k = \psi_{k+1}\circ \phi_k, \qquad k=1,2,\ldots. 
\end{equation}

Recall that a compact set $L$ in a complex manifold $X$ is said to be {\em holomorphically contractible}
if there exist a neighborhood $U\subset X$ of $L$ and a smooth $1$-parameter family of injective holomorphic maps
$F_t\colon U\to U$ $(t\in [0,1))$ such that $F_0$ is the identity map on $U$, $F_t(L)\subset L$ for every 
$t\in [0,1]$, and $\lim_{t\to 1} F_t$ is a constant map $L \mapsto p \in L$.

The first part of the following lemma is the key ingredient in the construction
of the sequence $(\phi_k)_{k\in \N}$ determining a long $\C^n$ as in Theorem \ref{th:main}.
The same construction gives the second part which we include for future applications.
We shall write $\C^*=\C\setminus\{0\}$.

%
%
%
%
\begin{lemma}\label{lem:mainstep}
Let $K$ be a compact set with nonempty interior in $\C^n$ for some $n>1$.
For every point $a\in \C^n$ there exists an injective holomorphic map $\phi\colon \C^n\hra\C^n$ such that  
the  polynomial hull of the set $\phi(K)$ contains the point $\phi(a)$.
More generally, if $L\subset \C^n$ is a compact holomorphically contractible set disjoint from $K$ 
such that $K\cup L$ is polynomially convex, then there exists an injective holomorphic map 
$\phi\colon \C^n\hra\C^n$ such that $\phi(L)\subset \wh{\phi(K)}$
and $\wh{\phi(K)} \setminus \phi(\C^n)\ne \emptyset$.
\end{lemma}

\begin{proof}
To simplify the notation, we consider the case $n=2$; it will be obvious that the same proof
applies in any dimension $n\ge 2$. We shall follow Wold's construction from \cite{Wold2008,Wold2010}
up to a certain point, adding a new twist at the end.

Let $M$ be a compact set in $\C^*\times \C$ enjoying the following properties:
\begin{enumerate}
\item $M$ is a disjoint union of two smooth, embedded, totally real discs;
\vspace{1mm}
\item $M$ is holomorphically convex in $\C^*\times \C$; 
\vspace{1mm}
\item the polynomial hull $\wh M$ of $M$ of contains the origin $(0,0)\in \C^2$.
\end{enumerate}
A set $M$ with these properties was constructed by Stolzenberg \cite{Stolzenberg1966};
it  has been reproduced in \cite[pp.\ 392--396]{Stout1971}, in \cite[Sec.\ 2]{Wold2008}, 
and in \cite[Section 4.20]{Forstneric2011}.

Choose a Fatou--Bieberbach map $\theta\colon \C^2\hra \C^*\times \C$ whose image $\theta(\C^2)$ is Runge in $\C^2$.
For example, we may take the basin of an attracting fixed point of a holomorphic automorphism
of $\C^2$ which fixes $\{0\}\times \C$ (cf.\ Rosay and Rudin \cite{RosayRudin1988}
for explicit examples). Replacing the set $K$ by its polynomial hull $\wh K$, we may assume that $K$ is polynomially convex.
Since $\theta(\C^2)$ is Runge in $\C^2$, the set  $\theta(K)$ is also polynomially convex, and hence 
$\Oscr(\C^*\times \C)$-convex. 
By \cite[Lemma 3.2]{Wold2008}, there exists a holomorphic automorphism $\psi\in \Aut(\C^*\times \C)$ such that 
\[
	\psi(M)\subset \theta(\mathring K). 
\]
The construction of such automorphism $\psi$ uses Theorem \ref{th:AL} applied with the manifold $X=\C^*\times \C$.
We include a brief outline. 

By shrinking each of the two discs in $M$ within themselves until they become very small 
and then translating them into $\mathring K$ within $\C^*\times \C$, 
we find an isotopy of diffeomorphisms $h_t\colon M=M_0\to M_t\subset \C^*\times \C$ $(t\in [0,1])$,
where each $M_t =h_t(M)$ is a totally real $\Oscr(\C^*\times \C)$-convex submanifold of $\C^*\times \C$, 
such that  $M_1\subset \mathring K$. Since $\C^*\times \C$ has the holomorphic density property (see Varolin \cite{Varolin2001}),
each diffeomorphism $h_t$  can be approximated uniformly on $M$ (and even in the smooth topology on $M$) 
by holomorphic automorphisms of $\C^*\times \C$. This is done in two steps.
First, we approximate $h_t$ by a smooth isotopy of biholomorphic maps 
$f_t\colon U_0 \to U_t$ from a neighborhood $U_0$ of $M_0$ onto a neighborhood $U_t$ of 
$M_t$; this is done as in \cite{ForstnericLow1997}.
Since the submanifold $M_t$ is totally real and $\Oscr(\C^*\times \C)$-convex for each $t\in[0,1]$, 
we can arrange that the neighborhood $U_t$ is Runge in $\C^*\times  \C$ for each $t\in [0,1]$. 
Hence, Theorem \ref{th:AL} furnishes an automorphism $\psi\in\Aut(\C^*\times \C)$
which approximates the diffeomorphism $h_1\colon M\to M_1$ sufficiently closely such
that $\psi(M)\subset B$. 
It follows that the injective holomorphic map $\tilde \phi=\psi^{-1}\circ\theta\colon\C^2\hra \C^*\times \C$
satisfies $M \subset \tilde\phi(\mathring K)$. Note that $K':=\tilde\phi(K)$ is a compact
$\Oscr(\C^*\times \C)$-convex set which contains $M$ in its interior. Therefore, its polynomial hull 
$\wh K'$ contains a neighborhood of $\wh M$, and hence a neighborhood $V\subset \C^2$ 
of the origin $(0,0)\in \C^2$. We may assume that $\overline V\cap K'=\emptyset$.

Let $a\in \C^2$. If $\tilde \phi(a)\in \wh K'$, then we take $\phi=\tilde\phi$ and we 
are done. If this is not the case, we choose a point $a'\in V \cap (\C^*\times \C)$ and
apply Theorem \ref{th:AL} to find a holomorphic automorphism $\tau\in\Aut(\C^*\times \C)$ 
which is close to the identity map on $K'$ and satisfies $\tau(\tilde \phi(a))=a'$. 
Such $\tau$ exists since the union of $K'$
with a single point of $\C^*\times \C$ is $\Oscr(\C^*\times \C)$-convex, so it suffices 
to apply the cited result to an isotopy of injective holomorphic maps which is the identity near $K'$
and which moves $\tilde \phi(a)$ to $a'$ in $\C^*\times \C\setminus K'$.
Assuming that $\tau$ is sufficiently close to the identity map on $K'$,
we have $M\subset \tau(K')$, and hence $a'\in \wh M\subset \widehat{\tau(K')}$. 
Clearly, the map $\phi=\tau\circ\tilde\phi\colon \C^2\to \C^*\times \C$ satisfies
$\phi(a) \in\wh{\phi(K)}$. This proves  the first part of the lemma.

The second part is proved similarly.
Since the set $L':=\theta(L)\subset \C^*\times \C$ is holomorphically contractible and 
$K'\cup L'$ is $\Oscr(\C^*\times \C)$-convex, there exists an automorphism 
$\tau\in \Aut(\C^*\times \C)$ which approximates the identity map on $K'$
and satisfies $\tau(L') \subset V\cap (\C^*\times\C)$. (To find such $\tau$, we apply Theorem \ref{th:AL}
to a smooth isotopy $h_t\colon U\to h_t(U)\subset \C^*\times \C$ $(t\in [0,1])$ 
of injective holomorphic maps on a small neighborhood $U\subset \C^*\times \C$ 
of $K'\cup L'$ such that $h_0$ is the identity on $U$, 
$h_t$ is the identity near $K'$ for every $t\in[0,1]$, 
and $h_1(L')\subset V$. On the set $L'$, $h_t$ first squeezes $L'$ within itself almost 
to a point and then moves it to a position within $V$.
Clearly, such an isotopy can be found such that $h_t(K'\cup L')= K_t\cup h_t(L')$ is
$\Oscr(\C^*\times \C)$-convex for all $t\in [0,1]$.)
If $\tau$ is sufficiently close to the identity on $K'$, then the polynomial hull $\widehat{\tau(K')}$
still contains $V$, and hence $\tau(L')\subset V\subset \widehat{\tau(K')}$.
The map $\phi=\tau\circ\tilde\phi\colon \C^2\to \C^*\times \C$ 
then satisfies the desired conclusion.
\end{proof}

\begin{proof}[Proof of Theorem \ref{th:main}]
Pick a compact set $K\subset \C^n$ with nonempty interior and a countable dense sequence
$\{a_j\}_{j\in \N}$ in $\C^n$. Set $K_1=\wh K$. Lemma \ref{lem:mainstep} furnishes an  injective 
holomorphic map $\phi_1\colon \C^n\to \C^n$ such that 
\begin{equation}\label{eq:tohull}
	\phi_1(a_1)\in \wh{\phi_1(K_1)}  =: K_2.
\end{equation} 
Applying Lemma \ref{lem:mainstep} to the set $K_2$ and the point $\phi_1(a_2)\in \C^n$ gives
an injective holomorphic map $\phi_2\colon \C^n\hra \C^n$ such that
\[
	\phi_2(\phi_1(a_2)) \in \wh{\phi_2(K_2)} =: K_3. 
\]
From the first step we also have $\phi_1(a_1)\in K_2$, and hence 
$\phi_2(\phi_1(a_1)) \in K_3$.

Continuing inductively, we obtain a sequence $\phi_j\colon \C^n\hra \C^n$ of injective holomorphic maps  
for $j=1,2,\ldots$ such that, setting $\Phi_k=\phi_k\circ\cdots\circ\phi_1\colon\C^n\hra \C^n$, we have 
\begin{equation}\label{eq:crux}
	\Phi_k(a_j)\in \wh{\Phi_k(K)}\quad \text{for all $j=1,\ldots,k$}.
\end{equation}

In the limit manifold $X=\bigcup_{k=1}^\infty X_k$ 
(the long $\C^n$) determined by the sequence $(\phi_k)_{k=1}^\infty$, 
the $\Oscr(X)$-hull of the initial set $K \subset \C^n=X_1\subset X$ clearly contains the
set $\Phi_k(K) \subset X_{k+1}$ for each $k=1,2,\ldots$. 
(We have identified the $k$-th copy of $\C^n$ in the sequence with its image 
$\psi_k(\C^n)=X_k \subset X$.) It follows from 
\eqref{eq:crux} that the hull $\wh K_{\Oscr(X)}$ contains the set $\{a_j\}_{j\in \N}\subset \C^n =X_1$.
Since this set is everywhere dense in $\C^n$, every holomorphic function on $X$
is bounded on $X_1=\C^n$, and hence constant. By the identity principle, it follows that the
function is constant on all of $X$.

The same argument shows that  the plurisubharmonic hull $\wh K_{\Psh(X)}$ of $K$
contains the set $A_1:=\{a_j\}_{j\in \N}\subset \C^n \cong X_1$, and hence
every plurisubharmonic function $u\in \Psh(X)$ is bounded from above on $A_1$.
Since $A_1$ is the dense in $X_1$,  it follows that $u$ is bounded from above on 
$X_1$. (This is obvious if $u$ is continuous; the general case follows by observing
that $u$ can be approximated from above, uniformly on compacts in $X_1\cong \C^n$, 
by continuous plurisubharmonic functions.) It follows from Liouville's theorem for plurisubharmonic 
that $u$ is constant on $X_1$.
In order to ensure that $u$ is constant on each copy $X_k\cong \C^n$ $(k\in\N)$ in the 
given exhaustion of $X$, we modify the construction as follows. 
After choosing the first Fatou--Bieberbach map $\phi_1\colon \C^n\hookrightarrow \C^n$
such that $\phi_1(a_1)\in \wh{\phi_1(K)}$ (see \eqref{eq:tohull}), we choose a countable dense set 
$A'_2=\{a'_{2,1},a'_{2,2},\ldots\}$ in $\C^n\setminus \phi_1(\C^n)$
and set $A_2=\phi_1(A_1) \cup A'_2$ to get a countable dense set in $X_2\cong \C^n$.
Next, we find a Fatou--Bieberbach map $\phi_2\colon \C^n\hra \C^n$
such that the first two points $\phi_1(a_1),\phi_1(a_2)$ of the set $\phi_1(A_1)$, and also the first 
point $a'_{2,1}$ of $A'_2$, are mapped by $\phi_2$ into the polynomial hull of 
$\phi_2(\phi_1(K))$. We continue inductively. At the $k$-th stage of the construction
we have  chosen a Fatou--Bieberbach map $\phi_k\colon\C^n\hra \C^n$, 
and we take $A_{k+1}=\phi_k(A_k)\cup A'_{k+1}$ where  $A'_{k+1}$ is a countable dense set
in $\C^n\setminus \phi_k(A_k)$. In the manifold $X$ we thus get an increasing sequence 
$A_1\subset A_2\subset \cdots$ whose union $A:=\bigcup_{k=1}^\infty A_k$ is dense in $X$
and such that every point of $A$ ends up in the hull $\wh K_{\Oscr(X_k)}=\wh K_{\Psh(X_k)}$ 
for all sufficiently big $k\in\N$.
(See the proof of Theorem \ref{th:SC} for more details in a related context.)
Hence, the plurisubharmonic hull $\wh K_{\Psh(X)}$ contains the countable dense subset $A$ of $X$. 
We conclude as before that any plurisubharmonic function on $X$ is bounded on every 
$X_k\cong \C^n$, and hence constant.
\end{proof}


\begin{remark}[\Wprocess]  \label{rem:Wprocess}
The key ingredient in the proof of Lemma \ref{lem:mainstep} is the method,
introduced by E.\ F.\ Wold in \cite{Wold2008}, of stretching the image
of a compact set in $\C^*\times \C^{n-1}$ by an automorphism of $\C^*\times \C^{n-1}$ so that its image
swallows a compact set $M$ whose polynomial hull in $\C^n$ intersects the hyperplane $\{0\}\times \C^{n-1}$. 
This will henceforth be called the {\em \Wprocess}. A recursive application of this method, 
possibly at several places simultaneously
and with additional approximation of the identity map on a certain other compact polynomially
convex set (cf.\ Lemma \ref{lem:mainstep2}), causes the hulls of the respective
sets to reach out of all domains $X_k\cong \C^n$ in the exhaustion of $X$.
\end{remark}


\section{Construction of manifolds $X(B)$}
\label{sec:XB}

In this section, we construct long $\C^n$'s satisfying Theorems \ref{th:XB} and \ref{th:SC}. 

We begin by showing that the stable hull property of a compact set in a complex manifold $X$ 
(see Definition \ref{def:SHP}) is independent of the choice of exhaustion of $X$ by compact sets.

%
%
%
%
\begin{lemma}\label{lem:SHP}
Let $X=\bigcup_{j=1}^\infty K_j$, where $K_j\subset \mathring K_{j+1}$ is a sequence 
of compact sets. Let $B$ be a compact set in $X$. Assume that there exists an integer $j_0\in\N$ 
such that $B\subset K_{j_0}$ and
\begin{equation}\label{eq:stable1}
  	{\wh B}_{\Oscr(K_j)}= {\wh B}_{\Oscr(K_{j_0})}\quad  \text{for all $j\ge j_0$}.
\end{equation}
Then, $B$ satisfies the same condition with respect to any exhaustion of $X$ 
by an increasing sequence of compact sets.
\end{lemma}

\begin{proof}
Set $C:= {\wh B}_{\Oscr(K_{j_0})}$.
Let $(L_l)_{l\in \N}$ be another exhaustion of $X$ by compact sets satisfying
$L_l\subset \mathring L_{l+1}$ for all $l\in\N$.
Pick an integer $l_0\in\N$ such that $C \subset L_{l_0}$. 
Since both sequences $\mathring K_j$ and $\mathring L_l$ exhaust $X$, we can find sequences of integers 
$j_1<j_2<j_3<\cdots$ and $l_1<l_2<l_3<\cdots$ such that $j_0\le j_1$, $l_0\le l_1$, and
\[
	K_{j_0 }\subset L_{l_1} \subset K_{j_1}\subset L_{l_2} \subset K_{j_2}\subset L_{l_3} \subset\cdots
\]
From this and \eqref{eq:stable1} we obtain
\[
	C=\wh{B}_{\Oscr(K_{j_0})} \subset \wh{B}_{\Oscr(L_{l_1})}  \subset  \wh{B}_{\Oscr(K_{j_1})}=C
	\subset \wh{B}_{\Oscr(L_{l_2})} \subset  \wh{B}_{\Oscr(K_{j_2})}=C \subset\cdots.
\]
It follows that $\wh{B}_{\Oscr(L_{l_j})}=C$ for all $j\in\N$.
Since the sequence of hulls $\wh{B}_{\Oscr(L_{l})}$ is increasing with $l$, 
we conclude that 
\[
	 \wh{B}_{\Oscr(L_{l})}=C  \quad \text{for all $l\ge l_1$}. 
\]
Hence, $B$ has the stable hull property with respect to the exhaustion $(L_l)_{l\in\N}$ of $X$.
\end{proof}

\begin{remark}\label{rem:Steinexh}
If a complex manifold $X$ is exhausted by an increasing sequence  of Stein
domains $X_1\subset X_2\subset \cdots\subset \bigcup_{j=1}^\infty X_j =X$
{(this holds for example if $X$ is a long $\C^n$ or a short $\C^n$, where the latter 
term refers to a manifold exhausted by biholomorphic copies of the ball)},
then we can choose an exhaustion $K_1\subset K_2\subset \cdots \subset \bigcup_{j=1}^\infty K_j =X$
such that $K_j$ is a compact set contained in $X_j$ and $\wh{(K_j)}_{\Oscr(X_j)}=K_j$
for every $j\in\N$. If $K$ is a compact set contained in some $K_{j_0}$, then clearly 
$\wh K_{\Oscr(K_j)}=\wh K_{\Oscr(X_j)}$ for all $j\ge j_0$. In such case,
$K$ has the stable hull property if and only if the sequence of hulls  
$\wh K_{\Oscr(X_j)}$ stabilizes.
This notion is especially interesting for a long $\C^n$. Imagining the exhaustion 
$X_j \cong\C^n$ of $X$ as an increasing sequence of universes, the stable 
hull property means that $K$ only influences finitely many of these universes in a nontrivial
way, while a set without SHP has nontrivial influence on all subsequent universes. 
\qed\end{remark}

We shall need the following lemma which generalizes \cite[Lemma 3.2]{Wold2008}.

%
%
%
%
%
\begin{lemma}\label{lem:mainstep2}
Let $n>1$. Assume that $B$ is a compact polynomially convex set in $\C^n$,
$K_1,\ldots, K_m$ are pairwise disjoint compact sets with nonempty interiors 
in $\C^n\setminus B$ such that $B\cup (\bigcup_{j=1}^m K_j)$ is polynomially convex, 
and $\beta \subset \C^n\setminus \bigl(B\cup (\bigcup_{j=1}^m K_j)\bigr)$ is a finite set.
Then there exists a Fatou--Bieberbach map $\phi\colon\C^n\hra\C^n$ satisfying the following conditions:
\begin{itemize}
\item[\rm (i)]    $\wh{\phi(B)} = \phi(B)$;
\vspace{1mm}
\item[\rm (ii)]   $\wh{\phi(K_j)} \not\subset \phi(\C^n)$ for all $j=1,\ldots,m$;
\vspace{1mm}
\item[\rm (iii)]  $\phi(\beta) \subset  \wh{\phi(K_1)}$.
\end{itemize}
Furthermore, we can choose $\phi$ such that $\phi|_B$ is as close as desired to the identity map.
\end{lemma}

\begin{proof}
For simplicity of notation we give the proof for $n=2$; the same argument applies for any $n\ge 2$.

By enlarging $B$ slightly, we may assume that it is a compact strongly pseudoconvex and polynomially 
convex domain in $\C^n$. Choose a closed ball $\Bcal \subset \C^2$ containing $B$ in its interior. 
Let $\Lambda\subset \C^2\setminus \Bcal$ be an affine complex line. 
Up to an affine change of coordinates on $\C^2$ we may assume that
$\Lambda=\{0\}\times \C$. 

As in the proof of Lemma  \ref{lem:mainstep},  we find 
an injective holomorphic map $\theta_1 \colon \C^2\hra \C^*\times \C$ whose image  is 
Runge in $\C^2$, and hence the set $\theta_1(\Bcal)$ is polynomially convex. 
Since $\Bcal$ is contractible, we can connect the identity map on $\Bcal$ to $\theta_1|_{\Bcal}$ 
by an isotopy of biholomorphic maps $h_t\colon \Bcal\to\Bcal_t$ $(t\in[0,1])$ with Runge images in $\C^*\times\C$. 
Theorem \ref{th:AL} furnishes an automorphism $\theta_2\in \Aut(\C^*\times\C)$ 
such that $\theta_2$ approximates $\theta_1^{-1}$ on $\theta_1(\Bcal)$.
The composition $\theta=\theta_2\circ \theta_1\colon \C^2\hra\C^*\times\C$
is then an injective holomorphic map which is close to the identity on $\Bcal$. Assuming that the
approximation is close enough, the set $B':=\theta(B)$ is polynomially convex in view of Lemma \ref{lem:stability}.
 
Set $K=\bigcup_{j=1}^m K_j$, $K'_j=\theta(K_j)$ for $j=1,\ldots,m$, and $K'=\theta(K)= \bigcup_{j=1}^m K'_j$.
Note that  the set $B'\cup K'=\theta(B\cup K) \subset \C^*\times \C$ is $\Oscr(\C^*\times \C)$-convex.

Choose $m$ pairwise disjoint copies $M_1,\ldots, M_m\subset (\C^*\times \C) \setminus B'$ of 
Stolzenberg's compact set $M$ (cf.\ \cite{Stolzenberg1966})
described in the proof of Lemma \ref{lem:mainstep}.
Explicitly, each set $M_j$ is $\Oscr(\C^*\times \C)$-convex and its polynomial hull 
$\wh M_j$ intersects the complex line $\{0\}\times \C$ (which lies 
in the complement of $\theta(\C^2)$). 
By placing the sets $M_j$ sufficiently far apart and away from $B'$, we may assume 
that the compact set $B'\cup( \bigcup_{j=1}^m M_j)$ is $\Oscr(\C^*\times \C)$-convex.
Pick a slightly bigger compact set $B''\subset \theta(\C^2)$, containing $B'$ 
in its interior, such that the sets $B''\cup (\bigcup_{j=1}^m K'_j)$  and 
$B''\cup( \bigcup_{j=1}^m M_j)$ are still $\Oscr(\C^*\times \C)$-convex.


We claim that for every $\epsilon>0$ there is an automorphism $\psi\in\Aut(\C^*\times\C)$ such that
\begin{itemize}
\item[\rm (a)]  $|\psi(z)-z|<\epsilon$ for all $z\in B''$, and
\vspace{1mm}
\item[\rm (b)]  $\psi(M_j)\subset K'_j$ for $j=1,\ldots,m$.
\end{itemize}
%
%
To obtain such $\psi$, we apply the construction in the proof of Lemma \ref{lem:mainstep}
to find an isotopy of smooth diffeomorphisms 
\[
	h_t\colon M=\bigcup_{j=1}^m M_j \to M^t=h_t(M) \subset \C^*\times \C, \quad t\in [0,1],
\]
such that $h_0=\Id|_M$, the set $M^t=\bigcup_{j=1}^m h_t(M_j)$ consists of smooth totally real submanifolds, 
$B''\cap M^t=\emptyset$ for all $t\in[0,1]$, $B''\cup M^t$ is $\Oscr(\C^*\times \C)$-convex for all $t\in [0,1]$,
and $h_1(M_j)\subset \mathring K'_j$ for $j=1,\ldots,m$.
It follows  that $h_1$ can be approximated uniformly on $M$ by a holomorphic 
automorphism $\psi\in \Aut(\C^*\times\C)$ which at the same time approximates the identity map on $B''$.
(For the details in a similar context, see \cite[proof of Theorem 2.3]{ForstnericRosay1993} or
\cite[proof of Corollary 4.12.4]{Forstneric2011}.)
The injective holomorphic map 
\[
	\phi := \psi^{-1}\circ\theta\colon \C^2\hra \C^*\times \C
\]
then approximates the identity map on a neighborhood of $B$ and satisfies $M_j \subset \phi(K_j)$ for $j=1,\ldots,m$. 
It follows that 
\[
	\wh{\phi(K_j)} \cap (\{0\}\times \C)\ne\emptyset\ \  \text{for all $j=1,\ldots,m$}.
\]
If the approximation $\psi|_{B''}\approx \Id$ in part (a) is close enough, then the set $\phi(B)=\psi^{-1}(B')$
is still polynomially convex by Lemma \ref{lem:stability}. 
Clearly, $\phi$ satisfies properties (i) and (ii) Lemma \ref{lem:mainstep2}, and 
property (iii) can be achieved by applying Lemma \ref{lem:mainstep}.
\end{proof}

%
%
%
%

\begin{proof}[Proof of Theorem \ref{th:SC}]
%
%
%
%
\noindent {\em Proof of part (a).} 
Let $B$ be the given regular compact polynomially convex set in $\C^n$. To begin the induction,
set $B_1:=B\subset \C^n=X_1$ and choose a pair of disjoint countable set 
\begin{eqnarray*}
	A_1 &=& \{a_1^l : l\in\N\} \subset \C^n\setminus B_1,\qquad\qquad \overline A_1=\C^n\setminus \mathring B_1, \\
	\Gamma_1 &=& \{\gamma_1^l : l\in\N\} \subset \C^n\setminus (A_1\cup B_1),
	\quad \overline \Gamma_1=\C^n\setminus \mathring B_1.
\end{eqnarray*}
Let $\B(a_1^1,r_1)$ denote the 
closed ball of radius $r_1$ centered at $a_1^1$.  By choosing $r_1>0$ small enough 
we may ensure that $\B(a_1^1,r_1) \cap B_1=\emptyset$, $\gamma^1_1\notin \B(a_1^1,r_1)\cup B_1$,
and the set $\B(a_1^1,r_1)\cup B_1$ is polynomially convex (cf. Lemma \ref{lem:stability}). 
Lemma \ref{lem:mainstep2} furnishes an injective holomorphic map  $\phi_1\colon \C^n\hra\C^n$ 
such that the set $B_2:=\phi_1(B_1)\subset \C^n$ is polynomially convex, while the 
compact set 
\[
	C^{1}_{1,1}:=\phi_1(\B(a_1^1,r_1))\subset \C^n
\]
satisfies 
\[
	\wh{C^{1}_{1,1}} \setminus \phi_1(\C^n)\neq\emptyset\quad 
	\text{and}\quad \phi_1(\gamma^1_1)\in \wh{C^{1}_{1,1}}.
\]

We proceed recursively.  Suppose that for some $k\in\N$ we have found 
\begin{itemize}
\item injective holomorphic maps $\phi_1,\ldots, \phi_k\colon \C^n\hra \C^n$,
\vspace{1mm}
\item compact polynomially convex sets $B_1,B_2\ldots,B_{k+1}$ in $\C^n$ such that $B_{i+1}=\phi_{i}(B_{i})$ for $i=1,\ldots,k$,
\vspace{1mm}
\item countable sets $A_1,\ldots,A_k\subset\C^n$ such that for every $i=1,\ldots,k$ we have
\[
	A_i\subset \C^n\setminus B_i, \quad \overline A_i=\C^n\setminus \mathring B_i,\quad  
	A_i=\phi_{i-1}(A_{i-1}) \cup \{a_i^l : l\in \N\}
\]
(where we set $A_0=\emptyset$), 
\vspace{1mm}
\item countable sets $\Gamma_1,\ldots,\Gamma_k\subset\C^n$ such that for every $i=1,\ldots,k$ we have
\[
	\Gamma_i\subset \C^n\setminus A_i\cup B_i, \quad \overline \Gamma_i=\C^n\setminus \mathring B_i,\quad  
	\Gamma_i=\phi_{i-1}(\Gamma_{i-1}) \cup \{\gamma_i^l : l\in \N\}
\]
(where we set $\Gamma_0=\emptyset$), and 
\vspace{1mm}
\item numbers $r_1> \ldots > r_k>0$
\end{itemize}
such that, setting for all $(i,l)\in\N^2$ with $1\le i+l\le k+1$:
%
%
\begin{eqnarray*}
	b_{k,i}^l:=\phi_{k-1} \circ\ldots\circ\phi_i(a_i^l)\in A_k \ \ \text{if}\ (i,l)\ne (k,1),
	\quad b_{k,k}^1:=a_k^1, \\
	\beta_{k,i}^l:=\phi_{k-1} \circ\ldots\circ\phi_i(\gamma_i^l)\in \Gamma_k \ \ \text{if}\ (i,l)\ne (k,1),
	\quad \beta_{k,k}^1:=\gamma_k^1,
\end{eqnarray*}
the following conditions hold  for all pairs $(i,l)\in\N^2$ with $i+l\le k+1$:
%
%
%
\begin{itemize}
\item[$(1_k)$] the closed balls $\B(b_{k,i}^l,r_{k})$ are pairwise disjoint and contained in $\C^n\setminus B_k$,
%
%
and $\{\beta_{k,i}^l:i+l\le k+1\} \cap  \bigcup_{i+l\le k+1} \B(b_{k,i}^l,r_{k}) =\emptyset$
(since $A_k\cap \Gamma_k=\emptyset$, the latter condition holds provided $r_k>0$ is small enough);
\vspace{1mm}
\item[$(2_k)$] the set $\bigcup_{i+l\le k+1} \B(b_{k,i}^l,r_{k}) \cup B_k$ is polynomially convex;
\vspace{1mm}
\item[$(3_k)$] the set $(\phi_{k-1}\circ\ldots\circ\phi_i)^{-1}(\B(b_{k,i}^l,r_{k}))$ is contained in $\B(a_i^l,{r_i}/{2^{k}})$;
\vspace{1mm}
\item[$(4_k)$]  the set $C^{l}_{k,i}:=\phi_{k}\bigl(\B(b_{k,i}^l,r_{k})\bigr)$ satisfies $\wh{C^{l}_{k,i}} \setminus \phi_{k}(\C^n) \neq\emptyset$;
\vspace{1mm}
\item[$(5_k)$]  $\{\phi_k(\beta_{k,i}^l):i+l\le k+1\} \subset  \wh{C^{1}_{k,1}}$.
\end{itemize}

We now explain the inductive step. We begin by adding to $\phi_k(A_k)$ countably many points in 
$\C^n\setminus (\phi_k(A_k)\cup B_{k+1})$ to get a countable set 
\[
	A_{k+1} = \phi_k(A_k) \cup\{a_{k+1}^l \colon l\in\N \}\subset\C^n\setminus B_{k+1}
\]
such that 
\[
	\overline A_{k+1} = \C^n\setminus \mathring B_{k+1}.
\] 
In the same way, we find the next countable set 
\[
	\Gamma_{k+1} = \phi_k(\Gamma_k) \cup\{\gamma_{k+1}^l \colon l\in\N \}\subset\C^n\setminus (A_{k+1}\cup B_{k+1})
\]
such that 
\[
	\overline  \Gamma_{k+1}  = \C^n\setminus \mathring B_{k+1}.
\] 
For every pair of indices $(i,l)\in\N^2$ with $i+l\le k+2$ we set 
\begin{eqnarray*}
	b_{k+1,i}^l:=\phi_k\circ\ldots\circ\phi_i(a_i^l)\in A_{k+1} \ \ \text{if}\ (i,l)\ne (k+1,1),
	\quad b_{k+1,k+1}^1:=a_{k+1}^{1}, \\
	\beta_{k+1,i}^l:=\phi_{k} \circ\ldots\circ\phi_i(\gamma_i^l)\in \Gamma_{k+1} \ \ \text{if}\ (i,l)\ne (k+1,1),
	\quad \beta_{k+1,k+1}^1:=\gamma_{k+1}^1.
\end{eqnarray*}
Choose a number $r_{k+1}$ with $0<r_{k+1}<r_k$ and so small that 
the following conditions hold for all $(i,l)\in\N^2$ with $i+l\le k+2$:
\begin{itemize}
\item[$(1_{k+1})$] the closed balls $\B(b_{k+1,i}^l,r_{k+1})$ are pairwise disjoint and contained in $\C^n\setminus B_{k+1}$, and
$\{\beta_{k+1,i}^l : i+l\le k+2\} \cap \bigl( \bigcup_{i+l\le k+2} \B(b_{k+1,i}^l,r_{k+1}) \cup B_{k+1}\bigr) =\emptyset$;
\vspace{1mm}
\item[$(2_{k+1})$] the set $\bigcup_{i+l\le k+2} \B(b_{k+1,i}^l,r_{k+1}) \cup B_{k+1}$ is polynomially convex;
\vspace{1mm}
\item[$(3_{k+1})$]  the set $(\phi_k\circ\ldots\circ\phi_i)^{-1}(\B(b_{k+1,i}^l,r_{k+1}))$ is contained in $\B(a_i^l,{r_i}/{2^{k+1}})$.
\end{itemize} 
Lemma \ref{lem:mainstep2} furnishes a Fatou--Bieberbach map     
$\phi_{k+1} \colon \C^n\hra\C^{n}$ such that the compact set $B_{k+2}:=\phi_{k+1}(B_{k+1})$
is polynomially convex, while the compact sets
\[
	C^{l}_{k+1,i}:=\phi_{k+1}\bigl(\B(b_{k+1,i}^l,r_{k+1})\bigr),\quad i+l\leq k+2 
\] 
satisfy the following conditions:
\begin{itemize}
\item[$(4_{k+1})$] $\wh{C^{l}_{k+1,i}} \setminus \phi_{k+1}(\C^n) \neq\emptyset$ for all $(i,l)\in\N^2$ with $i+l\leq k+2$;
\vspace{1mm}
\item[$(5_{k+1})$]  $\{\phi_{k+1}(\beta_{k+1,i}^l):i+l\le k+2\} \subset  \wh{C^{1}_{k+1,1}}$.
\end{itemize} 
This completes the induction step and the recursion may continue.

Let $X=\bigcup_{k=1}^\infty X_k$ be the long $\C^n$ determined by the sequence $(\phi_k)_{k=1}^\infty$. 
Since the set $B_k\subset \C^n$ is polynomially convex and $B_{k+1}=\phi_{k}(B_{k})$ for all $k\in\N$, 
the sequence $(B_k)_{k \in\N}$  determines a subset $B=B_1\subset X$ such that 
\begin{equation}\label{eq:hullB}
	\wh{B}_{\Oscr(X_k)}=B \quad \text{for all $k\in\N$}. 
\end{equation}
This means that the initial compact set $B\subset \C^n=X_1$ has the stable hull property in $X$. 

By the construction, the countable sets $A_k\subset \C^n\setminus B_k$ satisfy $\phi_k(A_k)\subset A_{k+1}$ for each $k\in\N$,
and hence they determine a countable set $A\subset X\setminus B$. Furthermore, since $\overline A_k = \C^n\setminus \mathring B_k$ 
for every $k\in\N$, it follows that $\overline A=X\setminus \mathring B$.
Similarly, the family $(\Gamma_k)_{k\in\N}$ determines a countable set $\Gamma\subset X\setminus B$
such that $\overline \Gamma =X\setminus \mathring B$.

We will now show that $B$ is the biggest regular compact set in $X$ with the stable hull property.
Note that Condition $(4_k)$, together with the fact that each 
set $C^{l}_{k,i}$ contains one of the sets $C^{l'}_{k+1,i'}$ in the next generation according to Condition $(3_{k+1})$
(and hence it contains one of the sets $C^{l'}_{k+j,i'}$ for every $j=1,2,\ldots$), implies
\begin{equation}\label{eq:Ukl}
	\wh{\bigl( C^{l}_{k,i}\bigr)}_{\Oscr(X_{k+j+1})} \setminus X_{k+j} \ne \emptyset
	\quad \text{for all $j=0,1,2,\ldots$}.
\end{equation}
Thus, none of the sets $C^{l}_{k,i}$ has the stable hull property.
Our construction ensures that the centers of these sets form a dense
sequence in $X\setminus B$, consisting of all points in the set $A$ determined by the family $(A_k)_{k\in\N}$,
in which every point appears infinitely often. Furthermore, Condition $(3_k)$ shows that 
every compact set $K\subset X$ with $\mathring K\setminus B\ne \emptyset$ 
contains one (in fact, infinitely many) of the sets $C^{l}_{k,i}$. In view of \eqref{eq:Ukl}, it
follows that there is an integer $k_0\in \N$ (depending on $K$) such that 
\[
	\wh K_{\Oscr(X_{k+1})} \not\subset X_k\quad  \text{for all}\ k\ge k_0.
\]
This means that $K$ does not have the stable hull property. 
It follows that the set $B$ is the strongly stable core of $X$. 

Finally, condition $(5_k)$ ensures that the $\Oscr(X)$-hull of a compact ball centered at the point $a^1_1\in A$ contains the 
countable set $\Gamma\subset X$ determined by the family $\{\Gamma_k\}_{k\in\N}$. 
Since $\Gamma$ is dense in $X\setminus B$, it follows that the manifold $X$ does not admit 
any nonconstant plurisubharmonic function. (See the proof of Theorem \ref{th:main} for the details.)

This proves part (a) of Theorem \ref{th:SC}.

\vspace{1mm}

\noindent {\em Proof of part (b).} Let $U\subset \C^n$ be an open set. 
Pick a regular compact polynomially convex set $B$ contained in $U$. We modify the recursion in the proof of part (a)
by adding to $B$ a new small closed ball $B'\subset U\setminus B$ at every stage.
In this way, we inductively build an increasing sequence $B=B^1\subset B^2\subset\cdots \subset U$
of compact polynomially convex sets whose union
$\Bcal := \bigcup_{k=1}^\infty B^k \subset U$ is everywhere dense in $U$,
and a sequence of Fatou--Bieberbach maps $\phi_k\colon\C^n\hra\C^n$
such that, writing 
\[
	B^k_1=B^k\quad\text{and}\quad \text{$B^k_{j+1}=\phi_{j}(B^k_{j})$\ \ for all $j,k\in\N$},
\]
the following two conditions hold:
\begin{itemize}
\item[\rm (a)] $B^k=B^{k-1} \cup \Bcal^k$ for all $k>1$, where $\Bcal^k$ is a small closed ball in $U\setminus B^{k-1}$;
\vspace{1mm}
\item[\rm (b)]  the set $B^k_{j}$ is polynomially convex for all $j,k\in\N$.
\end{itemize}
At the $k$-th stage of the construction we have already chosen Fatou--Bieberbach maps
$\phi_1,\ldots,\phi_k$, but we can nevertheless achieve Condition (b) for all $j=1,\ldots, k+1$ by choosing
the ball $\Bcal^k$ sufficiently small. Indeed, the image of a small ball by an injective
holomorphic map is a small strongly convex domain, and hence the polynomial 
convexity of the set $B^k_{j}$ for $j=1,\ldots, k+1$ follows from Lemma \ref{lem:stability}. 
For values $j> k+1$, Condition (b) is achieved by the construction in the proof of 
Lemma \ref{lem:mainstep2}; indeed, each of the subsequent maps $\phi_{k+1},\phi_{k+2},\ldots$ 
in the sequence preserves polynomial convexity of $B^k_{k+1}$.

By identifying the sets $U$ and $B^k=B^k_1$ (considered as subsets of $\C^n=X_1)$ 
with their images in the limit manifold $X=\bigcup_{k=1}^\infty X_k$, we thus obtain the following 
analogue of \eqref{eq:hullB}:
\[
	\wh{(B^k)}_{\Oscr(X_j)} = B^k \quad \text{for all $j,k\in\N$}. 
\]
This means that each set $B^k$ $(k\in\N)$ lies in the stable core $SC(X)$. 
Since $\bigcup_{k=1}^\infty B^k$ is dense in $U$ by the construction, we have that $\overline U\subset \overline{SC(X)}$.

On the other hand, writing $U_1=U$ and $U_{k+1}:=\phi_{k}\circ\cdots\circ \phi_1(U)$ for $k=1,2,\ldots$, 
the balls $\B(b_{k,i}^l,r_{k})$ chosen at the $k$-stage of the construction (see the proof of part (a))
are contained in $\C^n \setminus \overline U$ and, as $k$ increases, they
include more and more points from a countable dense set  $A\subset X\setminus \overline U$
which is built inductively as in the proof of part (a). By performing the \Wprocess\ on each of the balls
$\B(b_{k,i}^l,r_{k})$ (cf.\ condition $(4_k)$ above) at every stage, we can ensure that none of the points of $A$ belongs to 
the stable core $SC(X)$. Since $SC(X)$ is an open set by the definition and $\overline A=X\setminus  U$, 
we conclude that $SC(X)\subset U$. We have seen above that $\overline U\subset \overline{SC(X)}$, 
and hence $\overline{SC(X)}=\overline U$. 

%
%
It remains to show that $X$ can be chosen such that it does not admit any nonconstant holomorphic function.
By the same argument as in the proof of part (a), we can find a countable dense set
$\Gamma\subset X\setminus (A\cup \overline U)$ which is dense in $X\setminus U$ and is contained 
in the $\Oscr(X)$-hull of a certain compact set in $X\setminus \overline U$. It follows that every plurisubharmonic function $f$ on $X$
is bounded above on $\Gamma$, and hence on $\overline \Gamma=X\setminus U$.
If $\overline U$ is compact, the maximum principle implies that $f$ is also
bounded on $U$; hence it is bounded on $X$ and therefore constant. 
If $U$ is not relatively compact then we are unable to make this conclusion. 
However, we can easily ensure that $X\setminus \overline U$ contains a Fatou-Bieberbach domain; 
indeed, it suffices to choose the first Fatou-Bieberbach map $\phi_1\colon \C^n\to\C^n$ in 
the sequence determining $X$ such that $\C^n\setminus \phi_1(\C^n)$ contains a Fatou-Bieberbach domain 
$\Omega$. In this case, every holomorphic function $f\in\Oscr(X)$ is bounded on $\Gamma$, and hence on $\Omega$,
so it is constant on $\Omega\cong\C^n$. Therefore it is constant on $X$ by the identity principle. 

This proves part (b) and hence completes the proof of Theorem \ref{th:SC}.
\end{proof}


\section{An exhaustion of $\C^2$ by non-Runge Fatou--Bieberbach domains} 
\label{sec:exhaustion}

In this section, we show the following.

%
%
\begin{proposition}\label{prop:nonRunge} 
Let $n>1$. There exists an increasing sequence 
$X_1\subset X_2\subset \cdots\subset \bigcup_{k=1}^\infty X_k=\C^n$
of Fatou--Bieberbach domains in $\C^n$ which are not Runge in $\C^n$.
\end{proposition}

We shall construct such an example by ensuring that all Fatou--Bieberbach maps 
$\phi_k\colon \C^n\hra\C^n$ in the sequence (see Section \ref{sec:proof1}) 
have non-Runge images, but they approximate the identity map on increasingly large balls 
centered at the origin. For this purpose, we shall need the following lemma.

%
%
%
%
\begin{lemma}\label{lemmarunge} 
Let $B$ and $\Bcal$ be a pair of closed disjoint balls in $\C^n$ $(n>1)$.
For every $\epsilon>0$ there exists a Fatou--Bieberbach map 
$\phi \colon \C^n\hra\C^n$ satisfying the following conditions:
\begin{itemize}
\item[\rm (a)]  $||\phi|_B-\Id||<\varepsilon$;
\vspace{1mm}
\item[\rm (b)]  $||\phi^{-1}|_B - \Id||<\varepsilon$; 
\vspace{1mm}
\item[\rm (c)]  $\phi(\Bcal)$ is not polynomially convex.
\end{itemize}
\end{lemma}

\begin{proof}
Pick a slightly bigger ball $B'$ containing $B$ in the interior such that $B'\cap \Bcal=\emptyset$.
By an affine linear change of coordinates,  we may assume that $B'\subset \C^* \times\C^{n-1}$.
Choose a Fatou-Bieberbach map $\theta\colon\C^n\hra \C^*\times \C^{n-1}$ such that $\theta|_{B'}$
is close to the identity. (See the proof of Lemma \ref{lem:mainstep2}.) 
Theorem \ref{th:AL} provides a  $\psi\in\Aut(\C^*\times\C^{n-1})$ which
approximates the identity map on $\theta(B')$ and such that $\psi(\theta(\Bcal))$ is not polynomially convex
(in fact, its polynomial hull intersects the hyperplane $\{0\}\times\C^{n-1}$).
The composition $\phi=\psi\circ\theta\colon \C^n\hra \C^*\times\C^{n-1}$ then satisfies
Condition (a) on $B'$, and Condition (c). If $\phi$ is sufficiently close
to the identity on $B'$, then it also satisfies Condition (b) since $B\subset \mathring B'$.
\end{proof}

\begin{proof}[Proof of Proposition \ref{prop:nonRunge}]
Let $B_k=\B(0,k)\subset \C^n$ denote the closed ball of radius $k$ centered at  the origin.
Choose an integer $n_1\in\N$ and a small ball $\Bcal^1$ disjoint from $B_{n_1}$. 
Let $\phi_1\colon\C^n\to \C^n$ be a Fatou--Bieberbach map satisfying 
the following conditions:
\begin{enumerate}
\item $||\phi_1-\Id||<\varepsilon_1$ on $B_{n_1}$;
\vspace{1mm}
\item $||\phi_1^{-1}-\Id||<\varepsilon_1$ on $B_{n_1}$;
\vspace{1mm}
\item $\phi(\Bcal^1)$ is not polynomially convex.
\end{enumerate}
Suppose inductively that for some $k\in\N$ we have already found 
Fatou--Bieberbach maps $\phi_1,\ldots, \phi_k$, integers $n_1<n_2<\cdots <n_k$, 
and balls $\Bcal^j\subset \C^n\setminus B_{n_j}$ for $j=1,\ldots,k$ such that the following conditions hold:
\begin{itemize}
\item[$(1_{k})$]  $||\phi_{k}-\Id||<\varepsilon_{k}$ on $B_{n_{k}}$;
\vspace{1mm}
\item[$(2_{k})$]  $||\phi_{k}^{-1}-\Id||<\varepsilon_{k}$ on $B_{n_{k}}$;
\vspace{1mm}
\item[$(3_{k})$]  $\phi_{k}(\Bcal^{k})$  is not polynomially convex.
\end{itemize}
Choose an integer $n_{k+1}>n_{k}$ such that 
\[
	\phi_k(B_{n_k+1})\cup \phi_k(\phi_{k-1}(B_{n_{k-1}+2})) \cup 
	\ldots \cup\phi_k(\cdots(\phi_1(B_{n_1+k})))\cup\phi_k(\Bcal^k)\subset B_{n_{k+1}}
\]
and pick a ball $\Bcal^{k+1}\subset \C^n\setminus B_{n_{k+1}}$. 
By Lemma \ref{lemmarunge}, there exists a Fatou--Bieberbach map $\phi_{k+1}\colon \C^n\hra\C^n$  
satisfying the following conditions:
\begin{itemize}
\item[$(1_{k+1})$]  $||\phi_{k+1}-\Id||<\varepsilon_{k+1}$ on $B_{n_{k+1}}$;
\vspace{1mm}
\item[$(2_{k+1})$]  $||\phi_{k+1}^{-1}-\Id||<\varepsilon_{k+1}$ on $B_{n_{k+1}}$;
\vspace{1mm}
\item[$(3_{k+1})$]  $\phi_{k+1}(\Bcal^{k+1})$  is not polynomially convex.
\end{itemize}
This closes the induction step.

Let $X=\bigcup_{k=1}^\infty X_k$ be the long $\C^n$ determined by sequence $(\phi_k)_k$, let
$\iota_k \colon X_k\hra X_{k+1}$ denote the inclusion map, and let 
$\psi_k\colon \C^n \to X_k\subset X$ denote the biholomorphic map from $\C^n$ onto the 
$k$-th element of the exhaustion such that
\[
	\iota_k \circ \psi_k = \psi_{k+1}\circ \phi_k, \qquad k=1,2,\ldots. 
\]
(See Section \ref{sec:proof1}, in particular \eqref{eq:comm}.)
By the construction, the sequence $\psi_k(B_{n_k})$ is a Runge exhaustion of $X$. 
If the sequence $\varepsilon_k>0$ has been chosen to be summable, then 
the sequence $\psi_k$ converges on every ball $B_{n_j}$ and the limit
map $\Psi=\lim_{k\to\infty} \psi_k \colon \C^n \to X$ is a biholomorphism
(see \cite[Corollary 4.4.2, p.\ 115]{Forstneric2011}). 
In the terminology of Dixon and Esterle \cite[Theorem 5.2]{DixonEsterle1986},
we have that 
\[
	(\psi_k, B_{n_k}) \longrightarrow (\Psi,\C^n)\quad \text{as $k\to\infty$},
\]
where $\Psi(\C^n)=X$ and $\Psi$ is biholomorphic.
\end{proof}

\begin{remark}
If we only assume that the images of Fatou--Bieberbach maps $\phi_k\colon\C^n\hra \C^n$ 
contain large enough balls centered at the origin, we  get an exhaustion of a long 
$\C^n$ with Runge images of balls. By \cite[Theorem 3.4]{ArosioBracciWold2013}, 
such long $\C^n$ is biholomorphic to a Stein Runge domain in $\C^n$. 
Therefore, the following problem is closely related to Problem (C) stated in the Introduction.

\vspace{0.1cm}
\noindent 
(C') {\em If a long $\C^n$ is exhausted by Runge images of balls, is it necessarily biholomorphic to $\C^n$?}
\vspace{0.1cm}

In this connection, we mention that the first named author proved in his thesis \cite{Luka2016}
that $\C^n$ is the only Stein manifold with the density property
(see Definition \ref{def:DP}) having an exhaustion by Runge images of balls.
\end{remark}


\subsection*{Acknowledgements}
The authors are partially  supported by the research program P1-0291 and the grant J1-5432 from 
ARRS, Republic of Slovenia. A part of the work on this paper was done while F.\ Forstneri\v c was visiting
the Erwin Schr\"odinger Institute in Vienna, Austria, and of the Faculty of Science,
University of Granada, Spain, in November 2015. He wishes to thank both institutions for the invitation and excellent
working conditions. The author express their gratitude to anonymous referees for their valuable remarks
which helped us to improve and streamline the presentation.


\vskip 1cm
 
\noindent Luka Boc Thaler

\noindent Institute of Mathematics, Physics and Mechanics, Jadranska 19, SI--1000 Ljubljana, Slovenia


\noindent  e-mail: {\tt lukabocthaler@yahoo.com}

\vspace*{0.4cm}

\noindent Franc Forstneri\v c

\noindent Faculty of Mathematics and Physics, University of Ljubljana, and Institute
of Mathematics, Physics and Mechanics, Jadranska 19, SI--1000 Ljubljana, Slovenia

\noindent e-mail: {\tt franc.forstneric@fmf.uni-lj.si}

\end{document}